\documentclass{amsart}

\usepackage{amsmath}
\usepackage{amsfonts}
\usepackage{amssymb,enumerate}
\usepackage{amsthm}
\usepackage[all]{xy}
\usepackage{hyperref}

\newcommand{\Ann}{\operatorname{Ann}}
\newcommand{\Jac}{\operatorname{Jac}}
\newcommand{\Hom}{\operatorname{Hom}}

\newcommand{\Spec}{\operatorname{Spec}}

\newcommand{\depth}{\operatorname{depth}}

\renewcommand{\dim}{\operatorname{dim}}

\newcommand{\Max}{\operatorname{Max}}

\newcommand{\fn}{\frak{n}}
\newcommand{\fm}{\frak{m}}
\newcommand{\fp}{\frak{p}}

\newcommand{\fq}{\frak{q}}
\newtheorem{thm}{Theorem}[section]
\newtheorem{cor}[thm]{Corollary}
\newtheorem{lem}[thm]{Lemma}
\newtheorem{prop}[thm]{Proposition}

\newtheorem{exam}[thm]{Example}
\newtheorem{rem}[thm]{Remark}

\begin{document}

\bibliographystyle{amsplain}

\date{}

\author{P. Sahandi, N. Shirmohammadi and S. Sohrabi}

\address{Department of Mathematics, University of Tabriz,
Tabriz, Iran}
\email{sahandi@tabrizu.ac.ir}

\address{Department of Mathematics, University of Tabriz,
Tabriz, Iran} \email{shirmohammadi@tabrizu.ac.ir}

\address{Department of Mathematics, University of Tabriz,
Tabriz, Iran} \email{sohrabil6@yahoo.com}



%\thanks{Parviz Sahandi was in part supported by a grant from IPM (No.
%91130030)}

\title[Amalgamated algebra]{Cohen-Macaulay and Gorenstein properties under the amalgamated construction}

\begin{abstract}
Let $A$ and $B$ be commutative rings with unity, $f:A\to B$ a ring
homomorphism and $J$ an ideal of $B$. Then the subring
$A\bowtie^fJ:=\{(a,f(a)+j)|a\in A$ and $j\in J\}$ of $A\times B$ is
called the amalgamation of $A$ with $B$ along with $J$ with respect
to $f$. In this paper, among other things, we investigate the
Cohen-Macaulay and (quasi-)Gorenstein properties on the ring
$A\bowtie^fJ$.
\end{abstract}

\maketitle

\section{Introduction}

In \cite{DFF} and \cite{DFF2}, D'Anna, Finocchiaro, and Fontana have
introduced the following new ring construction. Let $A$ and $B$ be
commutative rings with unity, let $J$ be an ideal of $B$ and let
$f:A\to B$ be a ring homomorphism. They introduced the following
subring
$$A\bowtie^fJ:=\{(a,f(a)+j)|a\in A\text{ and }j\in J\}$$ of $A\times B$,
called the \emph{amalgamation of $A$ with $B$ along $J$ with respect
to $f$}. This construction generalizes the amalgamated duplication
of a ring along an ideal (introduced and studied in \cite{D},
\cite{DF}). Moreover, several classical constructions such as the
Nagata's idealization (cf. \cite[page 2]{Na}, \cite[Chapter VI,
Section 25]{Hu}), the $A + XB[X]$ and the $A+XB[[X]]$ constructions
can be studied as particular cases of this new construction (see
\cite[Examples 2.5 and 2.6]{DFF}).

Let $M$ be an $A$-module. In 1955, Nagata introduced a ring
extension of $A$ called the \emph{trivial extension} of $A$ by $M$
(or the \emph{idealization} of $M$ in $A$), denoted here by
$A\ltimes M$. Now, assume that $A$ is Noetherian local and that $M$
is finitely generated. It is well known that the trivial extension
$A\ltimes M$ is Cohen-Macaulay if and only if $A$ is Cohen-Macaulay
and the module $M$ is maximal Cohen-Macaulay. Furthermore, it is
proved by Reiten \cite{Re} and Foxby \cite{Fo} that $A\ltimes M$ is
Gorenstein if and only if $A$ is Cohen-Macaulay and $M$ is a
canonical module of $A$. Next, in \cite{A1}, Aoyama obtained a
generalization of this result to quasi-Gorenstein rings. Indeed, he
showed that $A\ltimes M$ is a quasi-Gorenstein ring if and only if
the completion $\widehat{A}$ satisfies Serre's condition $(S_2)$ and
$M$ is a canonical module of $A$.

Let $A$ be a Noetherian local ring and $I$ be an ideal of $A$.
Consider the amalgamated duplication $A\bowtie I:=\{(a,a+i)|a\in
A\text{ and }i\in I\}$ as in \cite{D}, \cite{DF}. D'Anna in \cite{D} (see also \cite{Sh})
proved that if $A$ is a Cohen-Macaulay local ring and $\Ann_A(I)=0$,
then the amalgamated duplication $A\bowtie I$ is Gorenstein if and
only if $A$ has a canonical ideal $I$. Next, in \cite{BSTY}, the
authors generalized this result as follows: if $\Ann_A(I)=0$, then
$A\bowtie I$ is a quasi-Gorenstein ring if and only if the
$\widehat{A}$ satisfies Serre's condition $(S_2)$ and $I$ is a
canonical ideal of $A$. Finally, in \cite{SSh}, the authors determined when the amalgamated duplication $A\bowtie I$ is normal and satisfies Serre's conditions $(S_n)$ and $(R_n)$.

In \cite[Remark 5.1]{DFF1}, assuming $A$ is a Cohen-Macaulay local
ring, $J$ is finitely generated as an $A$-module, and $J$ is contained in the  Jacobson radical of $B$, it is observed
that $A\bowtie^f J$ is a Cohen-Macaulay ring if and only if it is a
Cohen-Macaulay $A$-module if and only if $J$ is a maximal
Cohen-Macaulay module. Moreover, it is shown in \cite[Remark
5.4]{DFF1} that if $A$ is a Cohen-Macaulay local ring, having a
canonical module isomorphic (as an $A$-module) to $J$, then
$A\bowtie^f J$ is Gorenstein. Also, it is observed, under the
assumption $\Ann_{f(A)+J} (J) = 0$, that if $A$ is a Cohen-Macaulay
local ring and $A\bowtie^f J$ is Gorenstein, then $A$ has a
canonical module isomorphic to $f^{-1}(J)$ \cite[Proposition
5.5]{DFF1}.

The above results lead us to investigate further when the
amalgamated algebra $A\bowtie^fJ$ is Cohen-Macaulay or
(quasi-)Gorenstein.

More precisely, in Section 2, among other things, we prove that, for
a Noetherian local ring $(A,\fm)$, let $f:A\to B$ be a ring
homomorphism, and $J$ be an ideal of $B$, contained in the Jacobson
radical $B$, such that $f^{-1}(\fq)\neq\fm$, for each
$\fq\in\Spec(B)\backslash V(J)$ and that $\depth J<\infty$. Then
$A\bowtie^f J$ is Cohen-Macaulay if and only if $A$ is
Cohen-Macaulay and $J$ is a big Cohen-Macaulay module. In particular
if $J$ is finitely generated $A$-module (with the structure
naturally induced by $f$), then $A\bowtie^f J$ is Cohen-Macaulay if
and only if $A$ is Cohen-Macaulay and $J$ is maximal Cohen-Macaulay.
This improves \cite[Remark 5.1]{DFF1}, and \cite[Discussion 10]{D}.

In Section 3, among other things, we show that, for a local ring
$(A,\fm)$, $f:A\to B$ be a ring homomorphism, and $J$ be an ideal of
$B$ contained in the Jacobson radical $B$, such that $J$ is a
finitely generated $A$-module, if $\widehat{A}$ satisfies Serre's
condition $(S_2)$ and $J$ is a canonical module of $A$, then
$A\bowtie^fJ$ is quasi-Gorenstein.  Further, assume that $J^2=0$ and
that $\Ann_A(J)=0$. If $A\bowtie^fJ$ is quasi-Gorenstein, then $A$
satisfies $(S_2)$ and $J$ is a canonical module of $A$. This
generalizes \cite[Remark 5.4 and Proposition 5.5]{DFF1},
\cite[Theorem 11]{D} and \cite[Theorem 3.3]{BSTY}.

In Section 4, we describe the amalgamated algebra as a quotient of a
polynomial ring. As a consequence, we derive a characterization of
universally catenary property of the amalgamated algebra.

Now, in the following proposition, we collect some of the main
properties of the amalgamated algebra $A\bowtie^fJ$ needed in the
present paper. We use it in this paper without comments.

\begin{prop}\label{P} Let $A, B$ be commutative rings with
unity, $f:A\to B$ be a ring homomorphism, and $J$ be an ideal of
$B$.
\begin{itemize}
\item [(1)] (\cite[Proposition 5.1(1)]{DFF})
Let $\iota_A:A\to A\bowtie^fJ$ be the natural ring
homomorphism defined by $\iota_A(a)=(a,f(a))$,
for all $a\in A$. Then $\iota_A$ is an embedding,
making $A\bowtie^fJ$ a ring extension of $A$.
\item [(2)] (\cite[Proposition 5.7(a)]{DFF})
If $J$ is a finitely generated $A$-module
(with the structure naturally induced by $f$),
then $A\bowtie^fJ$ is Noetherian if and only
if $A$ is Noetherian.
\item [(3)] (\cite[Lemma 2.3(4)]{DFF}) Let
$\iota_J:J\to A\bowtie^fJ$, defined by $j\mapsto(0,j)$
and $P_A:A\bowtie^fJ\to A$, be the natural projection
of $A\bowtie^fJ\subseteq A\times B$ into $A$.
Then the following is a split exact sequence of $A$-modules:
    $$0\longrightarrow J\stackrel{\iota_J}{\longrightarrow}A\bowtie^fJ\stackrel{P_A}{\longrightarrow}A\longrightarrow0.$$
\item [(4)] (\cite[Corollaries 2.5 and 2.7]{DFF1})
For $\fp\in\Spec(A)$ and $\fq\in\Spec(B)\backslash V(J)$, set
\begin{align*}
\fp^{\prime_f}:= & \fp\bowtie^fJ:=\{(p,f(p)+j)|p\in \fp, j\in J\}, \\[1ex]
\overline{\fq}^f:= & \{(a,f(a)+j)|a\in A, j\in J, f(a)+j\in \fq\}.
\end{align*}
Then, the following statements hold.
\begin{itemize}
\item [(i)] The prime ideals of $A\bowtie^fJ$
are of the type $\overline{\fq}^f$ and $\fp^{\prime_f}$, for $\fq$
varying in $\Spec(B)\backslash V(J)$ and $\fp$ in $\Spec(A)$.
\item [(ii)] $\Max(A\bowtie^fJ)=\{\fp^{\prime_f}|\fp\in\Max(A)\}\cup\{\overline{\fq}^f|\fq\in\Max(B)\backslash V(J)\}.$
\item [(iii)] $A\bowtie^fJ$ is a local ring if
and only if $A$ is local and $J\subseteq\Jac(B)$.
\end{itemize}
\item [(5)] (\cite[Proposition 2.9]{DFF1})
The following statements hold.
\begin{itemize}
\item [(i)] For any prime ideal $\fq\in\Spec(B)\backslash V(J)$,
the localization $(A\bowtie^fJ)_{\overline{\fq}^f}$ is canonically
isomorphic to $B_{\fq}$.
\item [(ii)] For any prime ideal
$\fp\in\Spec(A)\backslash V(f^{-1}(J))$, the localization
$(A\bowtie^fJ)_{\fp^{\prime_f}}$ is canonically isomorphic to
$A_{\fp}$.
\item [(iii)] Let $\fp$ be a prime ideal of $A$
containing $f^{-1}(J)$. Consider the multiplicative subset
$S_{\fp}:=f(A\backslash \fp)+J$ of $B$ and set $B_{S_{\fp}}:=S_{\fp}^{-1}B$ and
$J_{S_{\fp}}:=S_{\fp}^{-1}J$. If $f_{\fp}:A_{\fp}\to B_{S_{\fp}}$ is the ring homomorphism
induced by $f$, then the ring $(A\bowtie^fJ)_{\fp^{\prime_f}}$ is
canonically isomorphic to $A_{\fp}\bowtie^{f_{\fp}}J_{S_{\fp}}$.
\end{itemize}
\end{itemize}
\end{prop}

\section{Cohen-Macaulay property under amalgamated construction}

Let us fix some notation which we shall use frequently throughout
this section: $A, B$ are two commutative rings with unity, $A$ is
Noetherian, $f:A\to B$ is a ring homomorphism, and $J$ denotes an
ideal of $B$, such that $A\bowtie^f J$ is Noetherian.

In \cite[Remark 5.1]{DFF1}, assuming $A$ is a Cohen-Macaulay local
ring, $J$ is finitely generated as an $A$-module, and $J$ is contained in the  Jacobson radical of $B$, it is observed
that $A\bowtie^f J$ is a Cohen-Macaulay ring if and only if it is a
Cohen-Macaulay $A$-module if and only if $J$ is a maximal
Cohen-Macaulay module. Also, in \cite[Remark 5.2]{DFF1}, the authors
mentioned that, if $J$ is not finitely generated as $A$-module, it
is more problematic to find conditions implying $A\bowtie^f J$ is
Cohen-Macaulay. Our main result in this section improves their
observation as well as, in case that $J$ is not finitely generated, it provides conditions implying $A\bowtie^f J$ is Cohen-Macaulay.

To state the main result we need to introduce some terminology. Let
$(A,\fm)$ be a Noetherian local ring, and $N$ an $A$-module (not
necessarily finitely generated). The \emph{depth} of $N$ over $A$ is
defined by the non-vanishing of the local cohomology modules
$H^i_{\fm}(N)$, with respect to $\fm$, in the way that
$$\depth N:=\depth_AN:=\inf\{i|H^i_{\fm}(N)\neq0\},$$
see \cite[Section 9.1]{BH} and \cite{FI}.
When $N$ is finitely generated, $\depth N$ coincides with
the usual depth defined by the common length of the maximal
$N$-regular sequences in $\fm$ by \cite[Proposition 3.5.4(b)]{BH}.
By \cite[Exercise 9.1.12]{BH}, if $\fm N\neq N$, then $\depth N$ is
finite, and when $\depth N$ is finite, then $\depth N\leq\dim A$.

It is well known that, over a local ring $(A,\fm)$, a finitely
generated $A$-module $M$ is Cohen-Macaulay if and only if
$H^i_{\fm}(M)=0$ for all $i<\dim M$ (see \cite[Corollary
6.2.8]{BS}).

Recall that a finitely generated module $M$ over a Noetherian local
ring $(A,\fm)$ is called a \emph{maximal Cohen-Macaulay $A$-module} if
$\depth M=\dim A$. An $A$-module $N$ is said to be \emph{big
Cohen-Macaulay} if $\depth N = \dim A$. Such modules were
constructed by Hochster \cite{Ho}, when $A$ contains a field. The
reader is referred to \cite[Chapter 8]{BH} for details on the
existence and properties of big Cohen-Macaulay modules.

The following auxiliary lemma which is interesting in itself is of
fundamental importance in our study of Cohen-Macaulayness.

In the sequel, $\Jac(B)$ will denote the Jacobson radical of $B$.

\begin{lem}\label{dim}
Let $(A,\fm)$ be a local ring, and $J\subseteq\Jac(B)$ be an ideal
of $B$ such that $f^{-1}(\fq)\neq\fm$, for each
$\fq\in\Spec(B)\backslash V(J)$. Then $\dim A\bowtie^f J=\dim A$,
and $\depth A\bowtie^f J=\min\{\depth A,\depth J\}$.
\end{lem}
\begin{proof}
By \cite[Corollary 3.2]{DFF1}, we have
$\fm^{\prime_f}=\sqrt{\fm(A\bowtie^f J)}$. Using this, the
Independence Theorem of local cohomology \cite[Theorem 4.2.1]{BS}
yields the isomorphism $H^i_{\fm^{\prime_f}}(A\bowtie^f J)\cong
H^i_{\fm}(A\bowtie^f J)$ for each $i$. On the other hand, applying
the functor of local cohomology to the split exact sequence of
$A$-modules appeared in Proposition \ref{P}(3), one obtains the
isomorphism $H^i_{\fm}(A\oplus J) \cong H^i_{\fm}(A)\oplus
H^i_{\fm}(J)$ for each $i$. Thus, for each $i$, we have shown the
following isomorphism
$$H^i_{\fm^{\prime_f}}(A\bowtie^f J)\cong
H^i_{\fm}(A)\oplus H^i_{\fm}(J)$$ of $A$-modules. Using above
isomorphism together with Grothendieck's Vanishing Theorem
\cite[Theorem 6.1.2]{BS}, one concludes the first equality.\\
In order to prove the second equality, again, we use above
isomorphism. First, note that $H^i_{\fm}(A)=H^i_{\fm}(J)=0$ for all
$i<\depth A\bowtie^f J$. Hence $\depth A\bowtie^f J\leq\depth A$ and
$\depth A\bowtie^f J\leq\depth J$. Now set $n:=\depth A\bowtie^f J$
and consider
$$H^{n}_{\fm}(A)\oplus H^{n}_{\fm}(J)\cong H^{n}_{\fm^{\prime_f}}(A\bowtie^f J)\neq0$$
to obtain $H^{n}_{\fm}(A)\neq0$ or $H^{n}_{\fm}(J)\neq0$. In the
first case we deduce that $\depth A=\depth A\bowtie^f J\leq\depth
J$. Similarly, in the second case we have $\depth J=\depth
A\bowtie^f J\leq\depth A$. Therefore the second equality is
obtained.
\end{proof}

\begin{rem}
Let $A$ be a (not necessarily Noetherian) ring, and $J$ be an ideal
of $B$. It is observed, in
\cite[Propositions 4.1 and 4.2]{DFF}, that if $f$ is surjective, or $f$ is integral, then $\dim A\bowtie^f J=\dim A$.
\end{rem}

The next theorem is the main result of this section. With it, we not
only offer an application of the above lemma, but we also provide
more information about the Cohen-Macaulayness under the amalgamation than was given in \cite[Remark 5.1]{DFF1}.

\begin{thm}\label{cm}
With the assumptions of Lemma \ref{dim}, the following statements hold.
\begin{enumerate}
       \item If $A\bowtie^f J$ is Cohen-Macaulay, then so does $A$.
       \item Further assume that $\depth J<\infty$.
       Then $A\bowtie^f J$ is Cohen-Macaulay if and
only if $A$ is Cohen-Macaulay and $J$ is a big Cohen-Macaulay module.
\end{enumerate}
\end{thm}
\begin{proof}
(1) Using Lemma \ref{dim}, we have $\dim A=\dim A\bowtie^f J=\depth
A\bowtie^f J\leq\depth A\leq\dim A$. Hence $A$ is Cohen-Macaulay.\\
(2) Assume that
$A\bowtie^fJ$ is Cohen-Macaulay. Then, by (1), $A$ is
Cohen-Macaulay. To prove the big Cohen-Macaulayness of $J$, note
that we have $\dim A=\dim A\bowtie^f J=\depth A\bowtie^f J\leq\depth
J\leq\dim A$, where the last inequality comes from \cite[Exercise
9.1.12]{BH}. Thus $\depth J=\dim A$, that is, $J$ is big
Cohen-Macaulay. Conversely, assume that $A$ is Cohen-Macaulay and
that $J$ is a big Cohen-Macaulay module. It follows that
$H^i_{\fm}(A)=H^i_{\fm}(J)=0$ for all $i<\dim A$. Then, by the above
isomorphism, one has $H^i_{\fm^{\prime_f}}(A\bowtie^fJ)=0$ for all
$i<\dim A=\dim A\bowtie^f J$, and hence $A\bowtie^fJ$ is
Cohen-Macaulay.
\end{proof}

The next example shows that, if, in the above theorem, the
hypothesis $f^{-1}(\fq)\neq\fm$, for each $\fq\in\Spec(B)\backslash
V(J)$, is dropped, then the corresponding statement is no longer
always true.

\begin{exam}
Let $k$ be a field and $X,Y$ are algebraically independent
indeterminates over $k$. Set $A:=k[[X]]$, $B:=k[[X,Y]]$ and let
$J:=(X,Y)$. Let $f: A\to B$ be the inclusion. Note that $A$ is
Cohen-Macaulay, and it is not hard to see that $J$ is a big
Cohen-Macaulay $A$-module. However, $A\bowtie^f J$ which is
isomorphic to $k[[X, Y,Z]]/(Y,Z)\cap(X-Y )$ is not Cohen-Macaulay.
It should be noted that, for $\fq:=(X)\in\Spec(B)\backslash V(J)$,
one has $f^{-1}(\fq)=(X)$ which is the maximal ideal of $A$.
\end{exam}

It is clear that a finitely generated big Cohen-Macaulay module is
maximal Cohen-Macaulay. Then we have the following corollary.

\begin{cor}\label{cm1}
Let $(A,\fm)$ be a local ring, and $J\subseteq\Jac(B)$ be an ideal
of $B$ such that $J$ is a finitely generated $A$-module. Then
$A\bowtie^f J$ is Cohen-Macaulay if and only if $A$ is
Cohen-Macaulay and $J$ is a maximal Cohen-Macaulay $A$-module.
\end{cor}
\begin{proof}
By \cite[Remark 3.3]{DFF1}, we see that if $J$ is finitely generated
$A$-module, then $f^{-1}(\fq)\neq\fm$, for each
$\fq\in\Spec(B)\backslash V(J)$. So the assertion holds by Theorem
\ref{cm}.
\end{proof}

Let $M$ be a finitely generated $A$-module. In 1955, Nagata
introduced a ring extension of $A$ called the \emph{trivial
extension} of $A$ by $M$ (or the \emph{idealization} of $M$ in $A$),
denoted here by $A\ltimes M$ (\cite[page 2]{Na}, \cite[Chapter VI,
Section 25]{Hu}). It should be noted that the module $M$ becomes an
ideal in $A\ltimes M$ and $M^2=0$. As in \cite[Example 2.8]{DFF}, if
$B:=A\ltimes M$, $J:=0\ltimes M$, and $f:A\to B$ be the natural embedding,
then $A\bowtie^f J\cong A\ltimes M$. Therefore the following result
follows from Corollary \ref{cm1}.

\begin{cor}\label{tc}
Let $(A,\fm)$ be a local ring, and $M$ be a finitely generated
$A$-module. Then the trivial extension $A\ltimes M$ is
Cohen-Macaulay if and only if $A$ is Cohen-Macaulay and $M$ is a
maximal Cohen-Macaulay $A$-module.
\end{cor}

\begin{cor}\label{cm2}
Assume that $f^{-1}(\fq)\neq\fm$, for each $\fq\in\Spec(B)\backslash
V(J)$ and each $\fm\in\Max(A)$. If $A\bowtie^f J$ is Cohen-Macaulay,
then so does $A$.
\end{cor}
\begin{proof}
Let $\fp\in\Max(A)$. If $\fp\notin V(f^{-1}(J))$, then
$A_{\fp}\cong(A\bowtie^f J)_{\fp^{\prime_f}}$ is Cohen-Macaulay. Now suppose $\fp\in
V(f^{-1}(J))$. Then
$A_{\fp}\bowtie^{f_{\fp}}J_{S_{\fp}}\cong(A\bowtie^f
J)_{\fp^{\prime_f}}$ is Cohen-Macaulay,
where $S_{\fp}:=f(A\backslash \fp)+J$. Note that, by \cite[Remark
2.4]{F}, one has $f_{\fp}^{-1}(\fq
B_{S_{\fp}})=f^{-1}(\fq)A_{\fp}\neq \fp A_{\fp}$, for each $\fq
B_{S_{\fp}}\in\Spec(B_{S_{\fp}})\backslash V(J_{S_{\fp}})$. Hence,
the Cohen-Macaulayness of $A_{\fp}$ results from Theorem \ref{cm}.
Therefore $A$ is Cohen-Macaulay.
\end{proof}

A finitely generated module $M$ over a Noetherian ring $A$ satisfies
Serre's condition $(S_n)$ if $\depth M_{\fp}\geq\min\{n,\dim
M_{\fp}\}$, for all $\fp\in\Spec(A)$. Note that if $M$ is
Cohen-Macaulay, then it satisfies Serre's condition $(S_n)$ for any
integer $n$. Also, when $\dim M=d$ and $M$ satisfies Serre's
condition $(S_d)$, then $M$ is Cohen-Macaulay. In the following
results we investigate the property $(S_n)$ for the amalgamated
algebra.

\begin{lem}\label{2}
Let $\fp\in\Spec(A)$, and set $S_{\fp}:=f(A\backslash \fp)+J$. If
$J^2=0$, then $J_{\fp}\cong J_{S_{\fp}}$ as $A_{\fp}$-modules.
\end{lem}
\begin{proof}
It is straightforward that the mapping $J_{\fp}\to J_{S_{\fp}}$
defined by $x/t\mapsto x/f(t)$, for all $x\in J$ and $t\in
A\backslash \fp$, is an $A_{\fp}$-isomorphism.
\end{proof}

Recall that an ideal is called a \emph{nil ideal} if each of its elements is nilpotent.

\begin{cor}\label{sur}
The following statements hold.
\begin{enumerate}
\item  Assume that for each $\fp\in V(f^{-1}(J))$
and each $\fq\in\Spec(B)\backslash V(J)$, $f^{-1}(\fq)\neq \fp$
(e.g. if $f$ is surjective or $J$ is a nil ideal of $B$). If $A\bowtie^f
J$ satisfies $(S_n)$, then so does $A$.
\item  Assume that $J^2=0$ and that $J$ is a finitely
generated $A$-module. If $A\bowtie^f J$ satisfies $(S_n)$, then so
does $J$.
\end{enumerate}
\end{cor}
\begin{proof}
(1) Let $\fp\in\Spec(A)$. If $\fp\notin V(f^{-1}(J))$, then, by
assumption, the inequality $\depth A_{\fp}\geq\min\{n,\dim
A_{\fp}\}$ holds for $A_{\fp}\cong(A\bowtie^f J)_{\fp^{\prime_f}}$.
Now suppose $\fp\in V(f^{-1}(J))$. This implies the isomorphism
$A_{\fp}\bowtie^{f_{\fp}}J_{S{\fp}}\cong(A\bowtie^f
J)_{\fp^{\prime_f}}$, where $S_{\fp}=f(A\backslash \fp)+J$.
Meanwhile, by \cite[Remark 2.4]{F}, we have the equality
$f_{\fp}^{-1}(\fq B_{S_{\fp}})=f^{-1}(\fq)A_{\fp}$, for each $\fq
B_{S_{\fp}}\in\Spec(B_{S_{\fp}})\backslash V(J_{S_{\fp}})$. Thus,
the assumptions of Lemma \ref{dim} hold. Therefore we obtain
\begin{align*}
\depth A_{\fp}\geq & \depth A_{\fp}\bowtie^{f_{\fp}}J_{S{\fp}} \\[1ex]
  = & \depth(A\bowtie^f J)_{\fp^{\prime_f}} \\[1ex]
  \geq & \min\{n, \dim (A\bowtie^f J)_{\fp^{\prime_f}}\} \\[1ex]
  = & \min\{n, \dim A_{\fp}\bowtie^{f_{\fp}}J_{S{\fp}}\} \\[1ex]
  = & \min\{n, \dim A_{\fp}\}.
\end{align*} \\
(2) Let $\fp\in\Spec(A)$. Then, by Lemma \ref{2}, we have the
isomorphism $J_{\fp}\cong J_{S_{\fp}}$ of $A_{\fp}$-modules. If
$\fp\notin V(f^{-1}(J))$, there is nothing to prove since
$J_{\fp}\cong J_{S_{\fp}}=0$ by \cite[Remark 2.4]{F}. So we can
assume that $\fp\in V(f^{-1}(J))$. Thus, using Lemma \ref{dim}, one
gets
\begin{align*}
\depth J_{\fp}\geq & \depth A_{\fp}\bowtie^{f_{\fp}}J_{\fp} \\[1ex]
  = & \depth(A\bowtie^f J)_{\fp^{\prime_f}} \\[1ex]
  \geq & \min\{n, \dim (A\bowtie^f J)_{\fp^{\prime_f}}\} \\[1ex]
  = & \min\{n, \dim A_{\fp}\bowtie^{f_{\fp}}J_{\fp}\} \\[1ex]
  = & \min\{n, \dim A_{\fp}\} \\[1ex]
  \geq & \min\{n, \dim J_{\fp}\}.
\end{align*}
Therefore $J$ satisfies $(S_n)$.
\end{proof}

One can employ Corollary \ref{sur} to deduce that the property
$(S_n)$ is retained under the trivial extension construction.

\begin{cor}
Let $(A,\fm)$ be a local ring, and $M$ be a finitely generated
$A$-module. Then the trivial extension $A\ltimes M$ satisfies
$(S_n)$ if and only if $A$ and $M$ satisfy $(S_n)$.
\end{cor}

\begin{prop}\label{sn}
If $A, B$ satisfy $(S_n)$, and $J_{S_{\fp}}$ is a maximal
Cohen-Macaulay $A_{\fp}$-module for each prime ideal $\fp$ of $A$,
where $S_{\fp}=f(A\backslash \fp)+J$, then $A\bowtie^f J$ satisfies
$(S_n)$.
\end{prop}
\begin{proof}
Each prime ideal of $A\bowtie^f J$ are of the type
$\overline{\fq}^f$ or $\fp^{\prime_f}$, for $\fq$ varying in
$\Spec(B)\backslash V(J)$ and $\fp$ in $\Spec(A)$. The localization
$(A\bowtie^f J)_{\overline{\fq}^f}$ is canonically isomorphic to
$B_{\fq}$. Thus one gets
$$
\depth (A\bowtie^f J)_{\overline{\fq}^f}=\depth
B_{\fq}\geq\min\{n,\dim B_{\fq}\}=\min\{n,\dim (A\bowtie^f
J)_{\overline{\fq}^f}\}.
$$
Next, for $\fp\in\Spec(A)\backslash V(f^{-1}(J))$, we have
$(A\bowtie^f J)_{\fp^{\prime_f}}\cong A_{\fp}$, hence, similarly,
the inequality $\depth (A\bowtie^f
J)_{\fp^{\prime_f}}\geq\min\{n,\dim (A\bowtie^f
J)_{\fp^{\prime_f}}\}$ holds. Now suppose $\fp\in V(f^{-1}(J))$.
This implies the isomorphism $(A\bowtie^f J)_{\fp^{\prime_f}}\cong
A_{\fp}\bowtie^{f_{\fp}}J_{S_{\fp}}$, where $S_{\fp}=f(A\backslash
\fp)+J$. Thus, by assumption, one deduces
\begin{align*}
\depth(A\bowtie^f J)_{\fp^{\prime_f}} = & \min\{\depth A_{\fp},\depth J_{S_{\fp}}\} \\[1ex]
  \geq & \min\{n, \dim A_{\fp}, \depth J_{S_{\fp}}\} \\[1ex]
  = & \min\{n, \dim A_{\fp}\} \\[1ex]
  = & \min\{n, \dim(A\bowtie^f J)_{\fp^{\prime_f}}\}.
\end{align*}
Therefore $A\bowtie^f J$ satisfies $(S_n)$.
\end{proof}

In concluding this section, we want to generalize the above theorem
to generalized Cohen-Macaulay rings. To this end, we need an
auxiliary lemma.

Let us recall that a finitely generated module $M$ over a Noetherian
local ring $(A,\fm)$ is said to be a \emph{generalized
Cohen-Macaulay} $A$-module if $H^i_\fm(M)$ is of finite length for
all $i<\dim M$. A local ring is called generalized Cohen-Macaulay if
it is a generalized Cohen-Macaulay module over itself. It is clear
that every Cohen-Macaulay module is a generalized Cohen-Macaulay
module.

In the course of next lemma and its proof, for a finite length
module $M$ over a ring $R$, we use $\ell_R(M)$ to denote the length
of $M$ over $R$.

\begin{lem}
Let $\varphi:(R,\fm)\to(S,\fn)$ be a local homomorphism of local
rings, such that the natural induced homomorphism
$\overline{\varphi}:R/\fm\to S/\fn$ is an isomorphism. Let $M$ be a
$S$-module. If $\ell_S(M)<\infty$, then $\ell_R(M)<\infty$ and
$\ell_S(M)=\ell_R(M)$. Here $M$ is considered as an $R$-module via
$\varphi$.
\end{lem}
\begin{proof}
Set $n:=\ell_S(M)$ and suppose that $0=M_0\subset
M_1\subset\cdots\subset M_n=M$ is a composition series of the
$S$-module $M$. Then $M_i/M_{i-1}\cong S/\fn$ for all
$i=1,\ldots,n$. On the other hand, it is clear that
$\overline{\varphi}$ is $R$-module homomorphism. This means
$S/\fn\cong R/\fm$ as $R$-modules. Thus, $M_i/M_{i-1}\cong R/\fm$
for all $i=1,\ldots,n$. Therefore $0=M_0\subset
M_1\subset\cdots\subset M_n=M$ is a composition series of $M$ as
$R$-module. This yields $\ell_R(M)<\infty$ and $\ell_R(M)=n$.
\end{proof}

\begin{thm}
Let $(A,\fm)$ be a local ring, and $J\subseteq\Jac(B)$ be an ideal
of $B$ such that $J$ is a finitely generated $A$-module. Then
$A\bowtie^fJ$ is a generalized Cohen-Macaulay ring if and only if
$A$ and $J$ are generalized Cohen-Macaulay and $\dim J\in\{0,\dim
A\}$.
\end{thm}
\begin{proof}
Suppose that $A\bowtie^fJ$ is generalized Cohen-Macaulay. Then
$H^i_{\fm^{\prime_f}}(A\bowtie^fJ)$ is of finite length over
$A\bowtie^fJ$ for all $i<\dim(A\bowtie^fJ)$. In conjunction with the
previous lemma, this shows that $H^i_{\fm^{\prime_f}}(A\bowtie^fJ)$
has finite length over $A$ for $i<\dim A$. Notice that, for all $i$,
one has the isomorphism $H^i_{\fm^{\prime_f}}(A\bowtie^fJ)\cong
H^i_{\fm}(A)\oplus H^i_{\fm}(J)$ of $A$-modules. Hence
$H^i_{\fm}(A)$ and $H^i_{\fm}(J)$ have finite length over $A$ for
all $i<\dim A$. Therefore $A$ and $J$ are generalized Cohen-Macaulay
and $\dim J=\dim A$ or 0 by \cite[Corollary 7.3.3]{BS}. Conversely,
suppose that $A$ and $J$ are generalized Cohen-Macaulay and $\dim
J\in\{0,\dim A\}$; so that there exists a positive integer $t$ such
that $\fm^t H^i_{\fm}(A)=0=\fm^tH^i_{\fm}(J)$ for all $i<\dim A$.
Hence $\fm^tH^i_{\fm^{\prime_f}}(A\bowtie^fJ)=0$ for all
$i<\dim(A\bowtie^fJ)$. On the other hand, by \cite[Corollary 3.2 and
Remark 3.3]{DFF1}, we know $\fm^{\prime_f}=\sqrt{\fm(A\bowtie^fJ)}$;
so that there exists a positive integer $s$ such that
$({\fm^{\prime_f}})^s\subseteq\fm(A\bowtie^fJ)$. Consequently
$({\fm^{\prime_f}})^{st}H^i_{\fm^{\prime_f}}(A\bowtie^fJ)=0$ for all
$i<\dim(A\bowtie^fJ)$. Therefore $A\bowtie^fJ$  is generalized
Cohen-Macaulay.
\end{proof}

\begin{cor}\label{tg}
Let $(A,\fm)$ be a local ring, and $M$ be a finitely generated
$A$-module. Then the trivial extension $A\ltimes M$ is generalized
Cohen-Macaulay if and only if $A$ and $M$ are generalized
Cohen-Macaulay and $\dim M\in\{0,\dim A\}$.
\end{cor}

Let $(A,\fm)$ be a Cohen-Macaulay local ring. Then, using
Corollaries \ref{tg} and \ref{tc}, the trivial extension of $A$ by
$A/\fm$ is a generalized Cohen-Macaulay ring which is not
Cohen-Macaulay.

\section{Gorenstein property under amalgamated construction}

Let $A, B$ be commutative rings with unity, $A$ be Noetherian,
$f:A\to B$ be a ring homomorphism, and $J\subseteq\Jac(B)$ denotes
an ideal of $B$, which is finitely generated $A$-module. It is shown
in \cite[Remark 5.4]{DFF1} that if $A$ is a local Cohen-Macaulay
ring, having a canonical module isomorphic (as an $A$-module) to
$J$, then $A\bowtie^f J$ is Gorenstein. Also it is observed, under
the assumption $\Ann_{f(A)+J} (J) = 0$, that if $A$ is a local
Cohen-Macaulay ring and $A\bowtie^f J$ is Gorenstein, then $A$ has a
canonical module isomorphic to $f^{-1}(J)$ \cite[Proposition
5.5]{DFF1}.

In present section, under some assumptions, specified in Theorem \ref{gor}, we give a sufficient
condition and a necessary condition for the ring $A\bowtie^f J$ to
be quasi-Gorenstein (definition recalled below). Consequently we improve the above mentioned
results.

We now outline to recall some terminology. Let $(A,\fm)$ be local. An $A$-module $K$ is
called a canonical module of $A$ if
$$K\otimes_A\widehat{A}\cong\Hom_A(H^{\dim A}_{\fm}(A),E_A(A/\fm)),$$
where $\widehat{A}$ is the $\fm$-adic completion of $A$, and
$E_A(A/\fm)$ is the injective hull of $A/\fm$ over $A$. We say that
$A$ is a \emph{quasi-Gorenstein ring} if a canonical module of $A$
exists and it is a free $A$-module (of rank one). This is equivalent
to saying that $H^{\dim A}_{\fm}(A)\cong E_A(A/\fm)$. It is known
that $A$ is quasi-Gorenstein if and only if $\widehat{A}$ is
quasi-Gorenstein \cite[Page 88]{A1}.

The following lemma is of crucial importance in this section.

\begin{lem}\label{cong}
With the notation of Proposition \ref{P}, the following statements
hold.
\begin{enumerate}
   \item If $\Ann_{f(A)+J}(J)=0$, then there is
   an $A\bowtie^fJ$-isomorphism $$\Hom_{A\bowtie^fJ}(A,A\bowtie^fJ)\cong f^{-1}(J)\times0.$$
   \item If $J^2=0$, then there is an $A$-isomorphism
   $$\Hom_{A\bowtie^fJ}(A,A\bowtie^fJ)\cong \Ann_A(J)\oplus J.$$
\end{enumerate}
\end{lem}
\begin{proof}
It is clear that $A\bowtie^fJ/0\times J\cong A$ as
$A\bowtie^fJ$-modules \cite[Proposition 5.1(2)]{DFF}. Then we have
\begin{align*}
\Hom_{A\bowtie^fJ}(A,A\bowtie^fJ)\cong & \Hom_{A\bowtie^fJ}(A\bowtie^fJ/0\times J,A\bowtie^fJ) \\[1ex]
  \cong & \Ann_{A\bowtie^fJ}(0\times J) \\[1ex]
  = & \{(a,f(a)+j)\in A\bowtie^fJ|(f(a)+j)J=0\} \\[1ex]
  = & \{(a,f(a)+j)\in A\bowtie^fJ|f(a)+j\in \Ann_{f(A)+J}(J)\}=:X.
\end{align*}
Now assume that $\Ann_{f(A)+J}(J)=0$. Then $X=f^{-1}(J)\times0$.
This completes the proof of (1). To prove (2), assume that $J^2=0$.
It follows that in this case $X=\{(a,f(a)+j)\in
A\bowtie^fJ|f(a)J=0\}$. Hence, one can define the $A$-homomorphism
$\psi:X\to\Ann_A(J)\oplus J$ by $\psi((a,f(a)+j))=(a,j)$. It is now
easy to check that $\psi$ is an isomorphism. Therefore (2) is
obtained.
\end{proof}

We are now ready for the proof of the main result of this section.

\begin{thm}\label{gor}
With the notation of Proposition \ref{P}, assume that $A$ is
Noetherian local, and $J\subseteq\Jac(B)$, which is finitely
generated $A$-module. The following statements hold.
\begin{enumerate}
 \item If $\widehat{A}$ satisfies $(S_2)$
and $J$ is a canonical module of $A$, then $A\bowtie^fJ$
is quasi-Gorenstein.
\item Assume that $J^2=0$ and that $\Ann_A(J)=0$. If $A\bowtie^fJ$
is quasi-Gorenstein, then $A$ satisfies $(S_2)$
and $J$ is a canonical module of $A$.
 \item Assume that $\Ann_{f(A)+J}(J)=0$ and that $A\bowtie^fJ$
is quasi-Gorenstein. Then $f^{-1}(J)$ is a canonical module of $A$.
Furthermore, assume that $f$ is surjective, then $\widehat{A}$ satisfies
$(S_2)$.
\item Assume that $J$ is a
  flat $A$-module. If $A\bowtie^fJ$
is quasi-Gorenstein, then $A$ is quasi-Gorenstein.
\end{enumerate}
\end{thm}
\begin{proof}
(1) Assume that $\widehat{A}$ satisfies $(S_2)$ and $J$ is a
canonical module of $A$. Using \cite[Satz 5.12]{HK}, this shows that
$A\bowtie^fJ$ has a canonical module
$K_{A\bowtie^fJ}\cong\Hom_A(A\bowtie^fJ,J)$. On the other hand,
applying the functor $\Hom_A(-,J)$ to the split exact sequence of
$A$-modules considered in Proposition \ref{P}(3), one obtains the
isomorphism $K_{A\bowtie^fJ}\cong \Hom_A(J,J)\oplus J$. Since
$\widehat{A}$ satisfies $(S_2)$, using \cite[Proposition 2]{A}, we
have $\Hom_A(J,J)\cong A$. Thus $K_{A\bowtie^fJ}\cong A\oplus J\cong
A\bowtie^fJ$ as $A$-modules. Note that
$K_{A\bowtie^fJ}\cong\Hom_A(A\bowtie^fJ,J)$ is an
$A\bowtie^fJ$-module by the usual way. Hence $K_{A\bowtie^fJ}\cong
A\bowtie^fJ$ as $A\bowtie^fJ$-modules. This means that $A\bowtie^fJ$
is quasi-Gorenstein.

(2) Assume that $A\bowtie^fJ$ is quasi-Gorenstein. Then by
\cite[Lemma 2.1]{HSZ}, we have $A\bowtie^fJ$ satisfies $(S_2)$.
Therefore by Corollary \ref{sur}, $A$ satisfies $(S_2)$. Since
$A\bowtie^{f}J$ is quasi-Gorenstein, thus a canonical module of
$A\bowtie^{f}J$ exists and it is isomorphic to $A\bowtie^{f}J$. By
\cite[Satz 5.12]{HK}, $A$ has a canonical module
$K_A\cong\Hom_{A\bowtie^fJ}(A,A\bowtie^fJ)$. On the other hand, by
Lemma \ref{cong}(2), we have $\Hom_{A\bowtie^fJ}(A,A\bowtie^fJ)\cong
J$. Therefore $J$ is a canonical module of $A$.

(3) First we show that $f^{-1}(J)$ is a canonical module of $A$. We
may use the same argument as employed in the proof of (2). One might
take into consideration that in the proof we appeal to Lemma
\ref{cong}(1) instead of Lemma \ref{cong}(2) in the context of the
proof of (2). Next, assuming the surjectivity of $f$, we show that
$\widehat{A}$ satisfies $(S_2)$. Since $A\bowtie^fJ$ is
quasi-Gorenstein, hence $\widehat{A\bowtie^fJ}$ is also
quasi-Gorenstein. Since $f$ is surjective, $B$ is a finitely
generated $A$-module, and therefore there is an
$\widehat{A}$-isomorphism
$\widehat{A\bowtie^fJ}\cong\widehat{A}\bowtie^{\widehat{f}}\widehat{J}$,
where $\widehat{f}:\widehat{A}\to\widehat{B}$ be the induced natural
surjective ring homomorphism. So that
$\widehat{A}\bowtie^{\widehat{f}}\widehat{J}$ is quasi-Gorenstein.
Then by \cite[Lemma 2.1]{HSZ}, we have
$\widehat{A}\bowtie^{\widehat{f}}\widehat{J}$ satisfies $(S_2)$.
Therefore by Corollary \ref{sur}, $\widehat{A}$ satisfies $(S_2)$.

(4) Assume that $A\bowtie^fJ$ is quasi-Gorenstein. Consider the
embedding $\iota_A:A\to A\bowtie^fJ$. Since $J$ is a flat
$A$-module, $\iota_A$ is a flat local homomorphism. Using
\cite[Remark 3.3]{DFF1}, the extension ideal $\fm(A\bowtie^fJ)$ is
$\fm^{\prime_f}=\fm\bowtie^fJ$-primary ideal. This in conjunction
with \cite[Theorem 2.3]{A1} implies that $A$ is quasi-Gorenstein.
\end{proof}

It is well known that $A$ is Gorenstein if and only if it is
Cohen-Macaulay and quasi-Gorenstein \cite[Page 88]{A1}. Also, it is
clear that $A$ is Cohen-Macaulay if and only if $\widehat{A}$ is
Cohen-Macaulay. Therefore, one can use Theorem \ref{cm} together
with Theorem \ref{gor} to derive the following corollaries.

\begin{cor}
Keep the assumptions of Theorem \ref{gor}. Assume that $A$ is
Cohen-Macaulay and $J$ is a canonical module of $A$. Then
$A\bowtie^fJ$ is Gorenstein.
\end{cor}

\begin{cor}
Keep the assumptions of Theorem \ref{gor}. Assume that at least one
of the following conditions holds
\begin{enumerate}
  \item $f$ is an isomorphism and $\Ann_B(J)=0$; or
  \item $J^2=0$ and $\Ann_A(J)=0$.
\end{enumerate}
Then $A\bowtie^fJ$ is Gorenstein if and only if $A$ is
Cohen-Macaulay and $J$ is a canonical module of $A$.
\end{cor}

\begin{cor}
Keep the assumptions of Theorem \ref{gor} and assume that
$\Ann_{f(A)+J}(J)=0$. If $A\bowtie^fJ$ is Gorenstein, then $A$ is
Cohen-Macaulay and $f^{-1}(J)$ is a canonical ideal of $A$.
\end{cor}

We recall that, if $\alpha:A\to C$, $\beta:B\to C$ are ring homomorphisms, the subring $D:=\alpha\times_C\beta:=\{(a,b)\in A\times B|\alpha(a)=\beta(b)\}$ of $A\times B$ is called the \emph{pullback} (or \emph{fiber product}) of $\alpha$ and $\beta$. We refer the readers attention to \cite{Fon}, \cite{Facc} and
in the Noetherian setting, to \cite{O}, \cite{O1} for the properties and importance of pullback constructions.
Let $f:A\to B$ be a ring homomorphism and $J$ be an ideal of $B$. If $\pi:B\to B/J$ is the
canonical projection and $\check{f}:=\pi\circ f$ , then $A\bowtie^fJ\cong\breve{f}\times_{B/J}\pi$, see \cite[Proposition 4.2]{DFF}.

 As an application of \cite[Theorem 4]{O1}, with some assumptions on the ideal $f^{-1}(J)$ and on the ring $f(A) + J$, the authors in \cite{DFF1} obtained the following characterization of when $A\bowtie^fJ$ is Gorenstein.

\begin{prop} (\cite[Proposition 5.7]{DFF1})
With the notation of Proposition \ref{P}, assume that $A$ is
Noetherian local, and $J\subseteq\Jac(B)$, which is finitely
generated $A$-module and, moreover, assume that $A$ is a Cohen-Macaulay ring,
$f(A)+ J$ is $(S_1)$ and equidimensional, $J\neq 0$ and that $f^{-1}(J)$ is a regular ideal
of $A$. Then, the following conditions are equivalent.
\begin{enumerate}
  \item $A\bowtie^fJ$ is Gorenstein.
  \item $f(A) + J$ is a Cohen-Macaulay ring, $J$ is a canonical module of $f(A) + J$ and $f^{-1}(J)$
is a canonical module of $A$.
\end{enumerate}
\end{prop}

In concluding this section, as an application, we provide a method to
construct a new quasi-Gorenstein ring from a given one. Let $A$ be a
local ring and $X$ be an indeterminate over $A$. Set
$B:=A[X]/(X^2)$, $J:=XB$, and $f:A\to B$ be the natural embedding.
It is easy to see that $J$ is isomorphic to $A$ as an $A$-module.
Then, using Theorem \ref{gor} (1) and (4), one sees that $A$ is
quasi-Gorenstein if and only if $A\bowtie^fJ$ is quasi-Gorenstein.
Note that we have $A\bowtie^fJ\cong A[X]/(X^2)$.

\section{Universally catenary property under amalgamated construction}

Let us fix some notation which we shall use frequently throughout
this section: $A, B$ are two commutative rings with unity, $f:A\to
B$ is a ring homomorphism, and $J$ denotes an ideal of $B$.

In \cite[Example 3.11]{DF} and \cite[Example 2.6]{DFF1}, the authors
in some special examples illustrate the amalgamated algebras as a
quotient of some polynomial rings. In the following proposition, we
generally describe the amalgamated algebra $A\bowtie^f J$ as a
quotient of a polynomial ring.

\begin{prop}\label{s}
Let the notation just introduced and assume that $J$ as an
$A$-module is generated by $\{j_{\lambda}|\lambda\in\Lambda\}$. Then
$A\bowtie^f J$ is a homomorphic image of
$A[X_{\lambda}|\lambda\in\Lambda]$ the ring of polynomials over $A$
in indeterminates $X_{\lambda}$, for each $\lambda\in\Lambda$.
\end{prop}
\begin{proof}
Consider the ring homomorphism
$\varphi:A[X_{\lambda}|\lambda\in\Lambda]\to A\bowtie^f J$, defined
by $\varphi(a):=(a,f(a))$, for $a\in A$, and
$\varphi(X_{\lambda}):=(0,j_{\lambda})$, for $\lambda\in\Lambda$. It
is not difficult to see that $\varphi$ is surjective.
\end{proof}

\begin{cor}
Assume that $I$ as an ideal of $A$ is generated by
$\{i_{\lambda}|\lambda\in\Lambda\}$. Then the amalgamated
duplication $A\bowtie I:=\{(a,a+i)|a\in A, i\in I\}$ (\cite{D, DF})
is a homomorphic image of $A[X_{\lambda}|\lambda\in\Lambda]$ the
ring of polynomials over $A$ in indeterminates $X_{\lambda}$, for
each $\lambda\in\Lambda$.
\end{cor}

As a consequence of the previous proposition, we provide a partial
characterization of the Noetherianity of $A\bowtie^f J$. It should
be noted that the authors in \cite[Proposition 5.7]{DFF} have
already given a characterization of the Noetherianity of $A\bowtie^f
J$. But, this is an obvious consequence of the above proposition
provided that $J$ is finitely generated as an $A$-module.

\begin{cor}
Assume that $J$ is finitely generated as an
$A$-module.
Then $A\bowtie^f J$ is Noetherian if and only if $A$ is Noetherian.
\end{cor}

The next corollary investigates the behaviour of universally
catenary property under amalgamated formation provided that $J$ is
finitely generated as an $A$-module. One says a locally finite
dimensional ring $A$ is \emph{catenary} if every saturated chain
joining prime ideals $\fp$ and $\fq$, $\fp\subseteq\fq$, has
(maximal) length height $\fq/\fp$; $A$ is \emph{universally
catenary} if all the polynomial rings $A[X_1,\ldots,X_n]$ are
catenary.

\begin{cor}
Assume that $J$ is finitely generated as an $A$-module. Then
$A\bowtie^f J$ is universally catenary if and only if $A$ is
universally catenary.
\end{cor}

\begin{center} {\bf ACKNOWLEDGMENT}

\end{center}
The authors are deeply grateful to the referee for his/her careful reading of the paper and valuable
suggestions.

\end{document}